\documentclass[review]{elsarticle}

\usepackage{lineno,hyperref}
\modulolinenumbers[1]

\usepackage{mystylefile}

\begin{document}
	
\begin{frontmatter}

\title{Multiplicative Perturbation Bounds for Multivariate Multiple Linear Regression in Schatten $p$-Norms}

\author{Jocelyn T. Chi\fnref{corres}}
\address{Department of Statistics\\ North Carolina State University\\ Raleigh, NC 27695, USA\\
	jtchi@ncsu.edu}
\fntext[corres]{Corresponding author}

\author{Ilse C. F. Ipsen}
\address{Department of Mathematics\\ North Carolina State University \\Raleigh, NC 27695, USA\\
	ipsen@ncsu.edu}


\begin{abstract}
Multivariate multiple linear regression (MMLR), which occurs in a number of practical applications, generalizes traditional least squares (multivariate linear regression) to multiple right-hand sides. We extend recent MLR analyses to sketched MMLR in general Schatten $p$-norms by interpreting the sketched problem as a multiplicative perturbation. Our work represents an extension of Maher's results on Schatten $p$-norms.  We derive expressions for the exact and perturbed solutions in terms of projectors for easy geometric interpretation.  We also present a geometric interpretation of the action of the sketching matrix in terms of relevant subspaces.  We show that a key term in assessing the accuracy of the sketched MMLR solution can be viewed as a tangent of a largest principal angle between subspaces under some assumptions.  Our results enable additional interpretation of the difference between an orthogonal and oblique projector with the same range.
\end{abstract}


\end{frontmatter}

%\linenumbers

\section{Introduction}

Multivariate multiple linear regression (MMLR)\footnote{We abbreviate multivariate multiple linear regression as ``MMLR" throughout this paper.} is a natural generalization of traditional least squares regression (multivariate linear regression) to multiple right-hand sides.  It is also useful in many large-scale real-world applications including image classification \cite{luo2014schatten, yang2016nuclear}, quality control monitoring \cite{eyvazian2011phase, noorossana2010statistical}, genetic association studies \cite{breiman1997predicting, li2015multivariate}, spatial genetic variation studies \cite{wang2013examining}, climate studies \cite{jeong2012multisite}, and low-rank tensor factorizations \cite{larsen2020practical} to name a few.  In the mathematics literature, least squares problems with multiple right-hand sides occur in the total least squares context, where both the independent and dependent variables may contain errors \cite{tls, core_lsmultRHS, optimal_lsmultRHS}.

In recent years, randomized approaches have become a popular method of dealing with very large data problems in numerical linear algebra \cite{mahoney2011randomized, woodruff2014sketching}.  The idea is to utilize random projections, random sampling, or some combination of the two to reduce the problem to a lower dimension while approximately retaining the characteristics of the original problem.  Referred to as \emph{sketching}, this has become a popular approach for the fast solution of highly overdetermined or underdetermined regression problems \cite{avron2010blendenpik, chi2018randomized, Drineas2011, MMY14, MMY15, LSRN, RM16, Sarlos2006}, where either the number of rows far exceeds the number of columns, or vice versa.

We view row-sketched MMLR as a multiplicative perturbation of MMLR, and derive perturbation bounds that are amenable to geometric interpretation.  Following up on our recent work \cite{chi2018randomized}, which quantifies the effect of sketching on the  geometry of traditional least squares, we extend our analysis to sketched MMLR in general Schatten $p$-norms. Our results represent an extension of Maher's work \cite{maher1990op, maher1992some, maher2007some, maher2007b} on Schatten $p$-norms.  
Schatten $p$-norms appear in numerous machine learning problems.  In particular, the nuclear ($p=1$) and Frobenius ($p=2$) norms appear in penalized regression \cite{wang2017sketched, yang2016nuclear}, regularized matrix regression \cite{zhou2014regularized}, matrix completion \cite{candes2010matrix, candes2009exact}, trace approximation \cite{han2017approximating, ubaru2017fast}, image feature extraction \cite{du2017two}, and image processing and classification \cite{lefkimmiatis2013hessian, wang2016schatten, wang2017optimal}.

\subsection{Problem setting}
We begin with the exact MMLR problem in a Schatten $p$-norm.  Denote the singular values of a matrix $\mm \in \real^{m \times d}$ by
\begin{eqnarray*}
	\sigma_{1}(\mm) \ge \sigma_{2}(\mm) \ge \cdots \ge \sigma_{\min(m,d)}(\mm) \ge 0.
\end{eqnarray*}
The Schatten $p$-norm \cite[page 199]{johnson1985matrix} of $\mm$ is a function of its singular values
\begin{eqnarray*}
\|\mm\|_{(p)} = \sqrt[\leftroot{-2}\uproot{2}p]{\sigma_{1}(\mm)^{p} + \cdots + \sigma_{r}(\mm)^{p}} \quad \text{ for } \quad 1 \le p \le \infty.
\end{eqnarray*}

Given a pair of matrices $\ma\in\real^{m\times n}$ and $\mb\in\real^{m\times d}$ with $\rank(\ma)=n$, the goal is to estimate the solution $\mxhat \in \real^{n \times d}$ satisfying
\begin{eqnarray}\label{eqn:exactMMLR}
\min_{\mx\in\real^{n \times d}}{\|\ma\mx-\mb\|_{(p)}} \quad \text{for } 1 \le p \le \infty.
\end{eqnarray}
Popular Schatten $p$-norms include the
\begin{itemize}
    \item $p=1$ nuclear (trace) norm $\lVert \mm \rVert_{*} = \sum_{j=1}^{\min(m,d)} \sigma_{j}(\mm) = \lVert \mm \rVert_{(1)}$,
    \item $p=2$ Frobenius norm $\lVert \mm \rVert_{\text{F}} = \sqrt{\sum_{j=1}^{\min(m,d)} \sigma_{j}(\mm)^{2}} = \lVert \mm \rVert_{(2)}$, $\quad$ and 
    \item $p = \infty$ Euclidean (operator) norm $\lVert \mm \rVert_{2} = \sigma_{1}(\mm) = \lVert \mm \rVert_{(\infty)}$.
\end{itemize}

Given a matrix $\ms \in \real^{c \times m}$ with $n \le c \le m$, the perturbed MMLR problem in a Schatten $p$-norm via randomized row-sketching is
\begin{eqnarray}\label{eqn:perturbedMMLR}
	\min_{\mx\in\real^{n \times d}}{\|\ms(\ma\mx-\mb)\|_{(p)}} \quad \text{for } 1 \le p \le \infty.
\end{eqnarray}
Row-sketching can be an effective approach to handling large data in the highly over-constrained case
\cite{DrineasMM2006, Drineas2011, MMY15, RM16, WZM18}, where $m \gg n$.  

\subsection{Existing work}

Widely considered to have originated in \cite{Sarlos2006}, randomized sketching has become a popular approach to solving large data problems in machine learning and numerical linear algebra \cite{mahoney2011randomized, woodruff2014sketching}.  In the regression setting, sketching approaches can be broadly classified \cite[Section 1]{THM2017} according to whether they achieve row compression \cite{BD09,DrineasMM2006,Drineas2011,ipsen2014effect,MMY14,MMY15,RT08,WZM18}, column compression \cite{avron2010blendenpik, THM2017}, or both \cite{LSRN}.

Recent work has improved the theoretical understanding of randomized regression from a statistical \cite{chi2018randomized, MMY14, MMY15, RM16, wang2017sketched} and geometric perspective \cite{chi2018randomized}.  Here, we extend the analysis in \cite{chi2018randomized} to the sketched MMLR problem in a Schatten $p$-norm.

The sketched MMLR problem in \eqref{eqn:perturbedMMLR} can be viewed as a generalization of weighted least squares since $\ms$ is not required to be positive definite diagonal \cite{MR2558818, MR976336, MR1745316}.  Additionally, \eqref{eqn:perturbedMMLR} holds more generally for Schatten $p$-norms with $1 \le p \le \infty$ rather than only the Frobenius norm.  Perturbation analysis for weighted least squares quantify the effect of additive perturbations of the weights, $\ma$, or both \cite{MR1745316}.  By constrast, we view the sketched problem in \eqref{eqn:perturbedMMLR} as a multiplicative perturbation of \eqref{eqn:exactMMLR}.

\subsection{Our contributions}

We show that the accuracy of the sketched MMLR solution in a Schatten $p$-norm depends on a term that captures both 1) how close the sketching matrix $\ms$ is to approximately preserving orthogonality \cite{chmielinski2005linear, mojvskerc2010mappings, turnvsek2007mappings} for any rank-preserving $\ms$ and 2) how close the columns of the sketched subspace are to being orthonormal (Proposition \ref{thm:error-rankpreserved}).  Our result is an extension of \cite[Lemma 1]{Drineas2011} as it holds under weaker assumptions and extends the result in \cite[Lemma 1]{Drineas2011} to the $d\ge 1$ case and for Schatten $p$-norms with $1 \le p \le \infty$.

We also present a geometric interpretation of the action of the sketching matrix $\ms$ in terms of relevant subspaces.  We show that a key term in assessing the accuracy of the sketched MMLR solution can be interpreted as the tangent of a largest principal angle between these subspaces if $\ms$ has orthonormal rows (Proposition \ref{prop:sqdagsqperp1}) or if $\ms$ preserves rank (Proposition \ref{prop:sqdagsqperp2}).  We then present a geometric interpretation of the operator norm difference between an orthogonal and oblique projector with the same range when $\ms$ preserves rank (Proposition \ref{prop:dist_Pa-P}).

\subsection{Preliminaries}
\label{sec:preliminaries}

We begin by setting some notation.  Let $\mi_{n} = \begin{pmatrix}\ve_{1} & \ve_{2} & \dots & \ve_{n} \end{pmatrix}$ denote the $n \times n$ identity matrix, and let the superscript $T$ denote the transpose.  
Let $\ma \in \real^{m \times n}$ be a matrix with $\rank(\ma) = n$.  
Then $\ma$ has the following full and thin QR decompositions
\begin{eqnarray}
\label{eqn:decompA}
\ma &= \begin{pmatrix}\mq & \mq_{\perp}\end{pmatrix} \begin{pmatrix} \mr \\ {\bf 0}_{(m-n) \times n} \end{pmatrix} = \mq \mr,
\end{eqnarray}
respectively, where $\mr \in \mathbb{R}^{n \times n}$ is nonsingular.  Thus, $\mq \in \real^{m \times n}$ and $\mq_{\perp} \in \real^{m \times (m-n)}$ represent orthonormal bases for $\range(\ma)$ and $\range(\ma)^{\perp} = \mynull(\ma\Tra)$, respectively. 

Since $\ma$ has full column rank, its Moore-Penrose generalized inverse is
\begin{align*}
\ma\Dag &= (\ma\Tra \ma)\Inv \ma\Tra %\\
= \mr\Inv \mq\Tra.
\end{align*}
The two-norm condition number of $\ma$ with respect to left inversion is 
\begin{eqnarray*}
	\kappa_{2}(\ma) &= \| \ma \|_{2}\, \|\ma\Dag\|_{2}.
\end{eqnarray*}

The following lemma asserts strong multiplicativity for Schatten $p$-norms and invariance under multiplication by matrices with orthonormal columns (rows) on the left (right).

\begin{lemma}[\protect{\cite[(2.7)]{maher2007b}}]
	\label{lem:pnorms}
	For $\ma \in \real^{m \times n}, \mb \in \real^{k \times m}$ and $\mc \in \real^{n \times l}$ with $1 \le p \le \infty$, we have
	\begin{eqnarray*}
		\| \mb \ma \mc \|_{(p)} \le \|\mb \|_{2} \|\ma \|_{2} \|\mc\|_{(p)}.
	\end{eqnarray*}
\end{lemma}
This version of Lemma \ref{lem:pnorms} is obtained from a modification of the proof for \cite[(2.5)]{maher2007b}.  Although we assume that $\rank(\ma) = n$ throughout this paper except in Proposition \ref{thm:orig}, Lemma \ref{lem:pnorms} holds regardless of whether or not $\ma$ has full column rank.

\section{Multivariate Multiple Linear Regression}\label{sec:problemsetup}

We describe the solution and regression residual for the exact and perturbed MMLR problems in a Schatten $p$-norm in \eqref{eqn:exactMMLR} and \eqref{eqn:perturbedMMLR}, respectively.  The following states that the solutions for \eqref{eqn:exactMMLR} are the same, regardless of the choice of $p \ge 1$ \cite{maher2007b}.

\begin{prop}[\cite{maher1990op, maher1992some, maher2007b}]\label{thm:orig}
	Let matrices $\ma \in \real^{m \times n}$ and $\mb \in \real^{m \times d}$ be given.   
	The MMLR problem in a Schatten $p$-norm
	\begin{eqnarray*}%\label{lem:exactMMLR}
		\min_{\mx\in\real^{n \times d}}{\|\ma\mx-\mb\|_{(p)}} \quad \text{ for } \quad 1 \le p \le \infty
	\end{eqnarray*}
	has the minimal Schatten $p$-norm solution $\mxhat \equiv \ma\Dag \mb$ with prediction and regression residual
	\begin{eqnarray*}
		\mbhat &\equiv& \ma \mxhat \quad \text{ and } \\
		\mrhat &\equiv& \mb - \ma\mxhat = (\mi - \ma \ma\Dag)\mb,
	\end{eqnarray*}
	respectively.  If $\rank(\ma) = n$, then the solution $\mxhat = \mr\Inv \mq\Tra \mb$ is unique with  
	regression residual $\mrhat = (\mi - \mq \mq\Tra)\mb = \mq_{\perp} \mq_{\perp}\Tra\mb$.
\end{prop}

\begin{proof}
For a proof that $\mxhat$ is the minimal Schatten $p$-norm solution to \eqref{eqn:exactMMLR}, see \cite{maher1990op, maher1992some, maher2007b}.  Specifically, \cite{maher1990op} shows that $\|\ma\mx - \mb\|_{(p)} \ge \|\ma\ma\Dag\mb - \mb\|_{(p)}$ for $2\le p < \infty$ and \cite{maher1992some} extends the result to $1 \le p < \infty$.  Then, \cite{maher2007b} extends the inequality to $1\le p\le \infty$ by showing that $\sigma_{j}(\ma\mx - \mb) \ge \sigma_{j}(\ma\ma\Dag\mb - \mb )$ for $j = 1, 2, \dots$ for finite rank operators.  Finally, \cite[Corollary 3.1]{maher2007b} shows that $\mxhat$ has minimal Schatten $p$-norm.  If $\rank(\ma) = n$, then $\mynull(\ma) = \{ {\bf 0} \}$ so that the general solution in \cite[Corollary 3.1]{maher2007b} is also unique. 
\end{proof}
	
Let $\ms \in \real^{c \times m}$ be a multiplicative perturbation matrix from the left with $n \le c \le m$ and $\rank(\ms\ma) \le \rank(\ma) = n$.  For example, $\ms$ may be a sampling matrix that extracts rows from $\ma$ \cite{Drineas2011, MMY15}, a projection matrix \cite{Ailon2009, Sarlos2006}, or a combination of sampling and projection matrices \cite{avron2010blendenpik, Drineas2011}.

\begin{prop}
	Let matrices $\ma \in \real^{m \times n}$ and $\mb \in \real^{m \times d}$ be given.   
	The perturbed MMLR problem in a Schatten $p$-norm
	\begin{eqnarray*}
		\min_{\mx\in\real^{n \times d}}{\|\ms(\ma\mx-\mb)\|_{(p)}} \quad \text{ for } \quad 1 \le p \le \infty
	\end{eqnarray*}
	in \eqref{eqn:perturbedMMLR} has the minimal Schatten $p$-norm solution $\mxtilde = (\ms\ma)\Dag\ms\mb$.  If $\rank(\ms\ma) = \rank(\ma)$ = n, then $\mxtilde$ is unique.
\end{prop}
Following convention \cite{MMY15, RM16}, we define the prediction  and regression residual of the perturbed MMLR problem to be
	\begin{eqnarray*}
		\mbtilde = \ma\widetilde{\mx} \quad \text{ and} \quad
		\mrtilde = \mb - \ma \widetilde{\mx}.
	\end{eqnarray*}
\section{General multiplicative perturbations}
\label{sec:genmultbounds}

We present general multiplicative perturbation bounds for \eqref{eqn:perturbedMMLR} requiring no assumptions on $\ms$.  To enable geometric interpretation, we express the bounds in terms of orthogonal and oblique projectors onto $\range(\ma)$ or a subspace of $\range(\ma)$.  For a matrix $\ma$, 
\begin{eqnarray*}
	\mPa = \ma\ma\Dag
\end{eqnarray*} 
denotes the orthogonal projector onto $\range(\ma)$ along $\mynull(\ma\Tra)$ (\cite[Theorem III.1.3]{MR1061154} and \cite{ChatterH86, HoagW78, VelleW81}).  For the perturbed MMLR problem in \eqref{eqn:perturbedMMLR}, 
\begin{eqnarray*}
	\mP \equiv \ma (\ms \ma)\Dag \ms
\end{eqnarray*}
denotes the corresponding oblique projector onto a subspace of $\range(\ma)$.  If $\rank(\ms\ma) = \rank(\ma)$, then $\range(\mP) = \range(\mPa)$ although $\mynull(\mP) = \mynull(\ma\Tra\ms\Tra\ms)$ \cite[Theorem 3.1]{cerny}, and $\mynull(\ma\Tra\ms\Tra\ms)\ne \mynull(\mPa)$ in general \cite[Lemma 3.1]{chi2018randomized}.  Oblique projectors appear in \cite{Ste2011, cerny} for constrained least squares, \cite{MR3021435} for discrete inverse problems, and \cite{brust2020computationally, MR976336} for weighted least squares.  The oblique projector $\mP$ can be viewed as an extension of the oblique projector
\begin{eqnarray*}
	\mP_{D} = \ma(\ma\Tra\md\ma)\Inv \ma\Tra\md
\end{eqnarray*}
in \cite{MR976336} if $\md = \ms\Tra \ms$ is a diagonal matrix with positive elements on the diagonal and $(\ma\Tra\md\ma)\Inv$ exists.  If $\ms$ is a sketching matrix that samples without replacement and $c = m$, then $\ms\Tra\ms=\mi_{m}$ satisfies the requirements for $\md$ in \cite{MR976336}.  In this case, however, the sketched MMLR problem in \eqref{eqn:perturbedMMLR} becomes the exact MMLR problem in \eqref{eqn:exactMMLR}.  If $d=1$ and $p=2$ in \eqref{eqn:perturbedMMLR}, the oblique projector $\mP$ appears in \cite{RM16} if $\rank(\ms\ma) = \rank(\ma)$ and in \cite[Lemma 3.1]{chi2018randomized} for any sketching matrix $\ms$.  Oblique projectors also appear in other problems, such as the discrete empirical interpolation method (DEIM) oblique projector $\mathbb{D}=\mU_{r}(\ms\Tra\mU_{r})\Dag \ms\Tra$ in \cite[Section 3.1]{MR3826673}.

Since $\ma\Dag$ is a left inverse of $\ma$, the exact and perturbed solutions are $\mxhat = \ma\Dag \mPa \mb$ and $\mxtilde = \ma\Dag\mP\mb$, respectively \cite[Lemma 3.1]{chi2018randomized}.  Therefore, the absolute error between the solution and regression residual are 
\begin{eqnarray*}
	\mxtilde - \mxhat &=& [(\ms \ma)\Dag \ms - \ma\Dag]\mb \;=\; \ma\Dag(\mP - \mPa)\mb \quad \text{and} \\
	\mrtilde - \mrhat &=& \ma [\ma\Dag - (\ms \ma)\Dag \ms]\mb \;=\; (\mPa - \mP)\mb.
\end{eqnarray*}

Proposition \ref{cor:error-norankpreserved} bounds the absolute error of the perturbed solution and regression residual for the MMLR problem in a Schatten $p$-norm with $1 \le p \le \infty$ in terms of the above projection matrices.
\begin{prop}
	\label{cor:error-norankpreserved}
	For the perturbed MMLR problem in \eqref{eqn:perturbedMMLR}, the absolute error bounds on the solution and regression residual in a Schatten $p$-norm are
 	\begin{eqnarray*}
	\|\mxtilde - \mxhat\|_{(p)} &\le& \| \ma\Dag\|_{2} \, \|\mP - \mPa\|_{2} \, \|\mb\|_{(p)}  \quad \text{and}\\
	\|\mrtilde - \mrhat\|_{(p)} &\le& \| \mP - \mPa \|_{2} \, \| \mb\|_{(p)}.
	\end{eqnarray*}
	If $\ma\Tra\mb \ne {\bf 0}$, the relative error bound in a Schatten $p$-norm is
	\begin{eqnarray*}	\frac{\| \mxtilde - \mxhat\|_{(p)}}{\|\mxhat\|_{(p)}}
	 \le \kappa_{2}(\ma) \, \| \mP - \mPa \|_{2} \, \frac{\| \mb \|_{(p)}}{\|\ma\|_{2}\|\mxhat\|_{(p)}}. \end{eqnarray*}
\end{prop}

\begin{proof}
	Lemma \ref{lem:pnorms} implies the bounds for the absolute error in a Schatten $p$-norm.
\end{proof}

Proposition \ref{cor:error-norankpreserved}, which extends \cite[Corollary 3.5]{chi2018randomized} to multiple right-hand sides and Schatten $p$-norms with $1 \le p \le \infty$, shows that the accuracy of the sketched solution and regression residual depends on the operator norm projector difference $\|\mP-\mPa\|_{2}$.
\section{Multiplicative Perturbations that Preserve Rank}
\label{sec:multi_perbs_preserve_rank}

We present multiplicative perturbation bounds for \eqref{eqn:perturbedMMLR} that hold if $\rank(\ms\ma) = \rank(\ma)$.  We begin by rewriting the difference between $\mPa$ and $\mP$ in terms of an orthonormal basis for the column space of $\ma$.  Since $\rank(\ms\ma) = n$, $(\ms\ma)\Dag = \mr\Inv (\ms\mq)\Dag$ so that 
\begin{eqnarray*}
	\mPa - \mP = \mq\mq\Tra - \mq(\ms\mq)\Dag\ms.
\end{eqnarray*}
Although the results in this section require the additional assumption that $\rank(\ms\ma)=\rank(\ma)$, they enable geometric interpretation beyond the difference between the projectors $\mPa$ and $\mP$.  

\begin{prop}
	\label{thm:error-rankpreserved}
	For the perturbed MMLR problem in \eqref{eqn:perturbedMMLR}, if $\rank(\ms\ma) = \rank(\ma)$, the absolute error bound in a Schatten $p$-norm for $1 \le p \le \infty$ is
	\begin{eqnarray*}
	\| \mxtilde - \mxhat \|_{(p)} &\le& \|\ma\Dag\|_{2} \,  \| (\ms\mq)\Dag \ms \mq_{\perp} \|_{2} \, \|\mrhat\|_{(p)}.
	\end{eqnarray*}
\end{prop}

\begin{proof}
	Since $\rank(\ms\ma) = n$, we have $(\ms\ma)\Dag = \mr\Inv(\ms\mq)\Dag$.  Thus,
	\begin{eqnarray}
		\mxtilde - \mxhat &=& (\ms\ma)\Dag\ms\mb - \ma\Dag\mb \notag\\
						&=& \mr\Inv[(\ms\mq)\Dag\ms - \mq\Tra]\mb. \label{eqn:mxtilde_mxhat}
	\end{eqnarray}
	Multiplying $\mb$ on the left by the identity matrix $\mi = \mq\mq\Tra + \mq_{\perp}\mq_{\perp}\Tra$ and inserting it in \eqref{eqn:mxtilde_mxhat} gives
	\begin{eqnarray}
		\mxtilde - \mxhat &=& \mr\Inv[(\ms\mq)\Dag\ms - \mq\Tra](\mq\mq\Tra + \mq_{\perp}\mq_{\perp}\Tra)\mb \notag\\
		&=& \mr\Inv(\ms\mq)\Dag\ms\mq_{\perp}\mq_{\perp}\Tra\mb. \label{eqn:likedrineaspaper}
	\end{eqnarray}
	
	Lemma \ref{lem:pnorms} and unitary invariance of the operator norm imply{\tiny } the following upper bound on the Schatten $p$-norm of the absolute error difference between the sketched and exact MMLR solutions
	\begin{eqnarray*}
	\| \mxtilde - \mxhat \|_{(p)} 
	&\le& \| \mr\Inv \|_{2} \, \| (\ms\mq)\Dag \ms \mq_{\perp} \|_{2} \, \| \mq_{\perp}\mq_{\perp}\Tra\mb \|_{(p)}.
	\end{eqnarray*}
	Finally, applying the definition of the exact regression residual $\mrhat = \mq_{\perp}\mq_{\perp}\Tra\mb$ concludes the proof.
\end{proof}

Since $\|\ma\Dag\|_{2}$ and $\|\mrhat\|_{(p)}$ are fixed for any pair of $\ma$ and $\mb$, only $\| (\ms\mq)\Dag \ms\mq_{\perp} \|_{2}$ is affected by the choice of the sketching matrix $\ms$.  
We compare this to the approximate isometry term $\| (\ms\mq)\Tra \ms \mrhat \|_{2}$ from \cite[Equation 9]{Drineas2011}, where $(\ms\mq)\Tra \ms \mrhat$ is a vector.  Notice that we can arrive at the $\| (\ms\mq)\Tra \ms \mrhat \|_{2}$ term if we revert to \eqref{eqn:likedrineaspaper} in the above proof and assume that the columns of $\ms\mq$ are orthonormal so that $(\ms\mq)\Dag = (\ms\mq)\Tra$.  If we further restrict our analysis to the $d=1$ and $p=2$ case, we recover the same normed quantity as in \cite[Equation 9]{Drineas2011}.  Thus, we compare Proposition \ref{thm:error-rankpreserved} to \cite[Lemma 1]{Drineas2011}, where the absolute solution error for the $d=1$ and $p=2$ case is 
	\begin{eqnarray} \label{eqn:Drineas2011}
	\| \mxhat - \widetilde{\mx} \|_{2} &\le& \|\ma\Dag\|_{2} \sqrt{\epsilon} \| \mrhat\|_{2}
	\end{eqnarray}
for $\epsilon$ and $\ms$ satisfying \cite[Equations 8 and 9]{Drineas2011}:
	\begin{eqnarray}
	\|(\ms \mq)\Dag\|_{2} &\le& 2^{\frac{1}{4}} \quad \text{ and} \label{eqn:Drineas2011rank} \\
	\| (\ms\mq)\Dag \ms\mrhat \|_{2} &\le& \sqrt{\frac{\epsilon}{2}} \| \mrhat\|_{2}, \label{eqn:Drineas2011orthogonal}
	\end{eqnarray}
	
Proposition \ref{thm:error-rankpreserved} can be viewed as an extension of \cite[Lemma1]{Drineas2011} in the following ways.  First, Proposition\ref{thm:error-rankpreserved} extends the result in \cite[Lemma 1]{Drineas2011} for $d \ge 1$ and for Schatten $p$-norms with $1 \le p \le \infty$.  Second, \cite[Lemma 1]{Drineas2011} is a special case of Proposition \ref{thm:error-rankpreserved} when $d=1$, $p=2$, and $\sqrt{\epsilon} = \| (\ms\mq)\Dag \ms \mq_{\perp} \|_{2}$.  Third, in contrast with \cite[Lemma 1]{Drineas2011}, the bound in Proposition \ref{thm:error-rankpreserved} holds without requiring the assumptions \eqref{eqn:Drineas2011rank} or \eqref{eqn:Drineas2011orthogonal}.
	
\section{Angle between the original and perturbed subspaces}
\label{sec:distance}
 
We show that $\|(\ms\mq)\Dag\ms\mq_{\perp}\|_{2}$ is the tangent of a largest principal angle under two conditions: if $\ms$ has orthonormal columns, or if $\ms$ preserves rank.  Furthermore we show that if $\ms$ preserves rank, then $\|(\ms\mq)\Dag\ms\mq_{\perp}\|_{2}$ also represents the operator norm difference between the orthogonal projector $\mPa$ and the oblique projector $\mP$.  Therefore, if an orthogonal and an oblique projector have the same range, then their operator norm difference can be interpreted in terms of principal angles.  We begin with a decomposition of $\range(\ms\Tra)$ with respect to $\range(\mq)$ and $\range(\mq_{\perp})$.

\subsection{A decomposition of $\range(\ms\Tra)$}\label{sec:decomp}

The following geometric interpretations depend on a decomposition of $\ms$ into three subspaces.  
Let $\mathcal{Q} \equiv \range(\mq)$, $\mathcal{Q}^{\perp} \equiv \range(\mq_{\perp})$, and $\mathcal{S} \equiv \range(\ms\Tra)$.  Following the notation in \cite[Section 2]{ZKny13}, we can decompose $\mathcal{S}$ into the direct sum of the following subspaces
\begin{eqnarray*}
	\mathcal{S}_{1} \equiv \mathcal{S} \cap \mathcal{Q}, \quad
	\mathcal{S}_{0} \equiv \mathcal{S} \cap \mathcal{Q}^{\perp}, \quad \text{ and } \quad
	\mathcal{S}_{10} \equiv \mathcal{S} \cap (\mathcal{Q} \oplus \mathcal{Q}^{\perp})^{\perp}.
\end{eqnarray*}
We summarize and interpret these subspaces of $\mathcal{S}$ as follows.  The subspace $\mathcal{S}_{1}$ contains the directions in $\mathcal{S}$ that are also in $\mathcal{Q}$.  Specifically, $\mathcal{S}_{1} = \{ \vs \in \mathcal{S}: \vs\Tra\vq = \|\vs\|_{2}\|\vq\|_{2} \text{ for some } \vq \in \mathcal{Q} \}$, where $\| \cdot \|_{2}$ denotes the Euclidean vector norm.  

The subspace $\mathcal{S}_{0}$ contains the directions in $\mathcal{S}$ that are also in $\mathcal{Q}^{\perp}$.  Therefore, these are the directions in $\mathcal{S}$ that are orthogonal to directions in $\mathcal{Q}$.  Specifically, $\mathcal{S}_{0} = \{ \vs \in \mathcal{S}: \vs\Tra\vq = 0 \text{ for all } \vq \in \mathcal{Q} \}$.  

The subspace $\mathcal{S}_{10}$ contains the directions in $\mathcal{S}$ that are in neither $\mathcal{Q}$ nor $\mathcal{Q}^{\perp}$.  Therefore, these are the directions in $\mathcal{S}$ that are not orthogonal to $\mathcal{Q}$ but are also not in $\mathcal{Q}$.  Specifically, $\mathcal{S}_{10} = \{ \vs \in \mathcal{S}: 0 < |\vs\Tra\vq| < \|\vs\|_{2}\|\vq\|_{2} \text{ for all } \vq \in \mathcal{Q} \}$.

The subspace
\begin{eqnarray*}
	\mathcal{S}_{Q} & \equiv & \mathcal{S}_{1} \oplus \mathcal{S}_{10},
\end{eqnarray*}
then comprises the directions in $\mathcal{S}$ that are not orthogonal with directions in $\mathcal{Q}$. Specifically, $\mathcal{S}_{Q} = \{ \vs \in \mathcal{S}: 0 < | \vs\Tra\vq | \le \|\vs\|_{2}\|\vq\|_{2} \text{ for all } \vq \in \mathcal{Q} \}$.  

Section \ref{sec:example} presents an illustrative example of these subspaces in the context of Proposition \ref{prop:sqdagsqperp2}.  
In general, we have
\begin{eqnarray*}
	\dim(\mathcal{S}_{1}) \le \dim(\mathcal{Q}) = n
\end{eqnarray*}
and
\begin{eqnarray*}
	\dim(\mathcal{S}_{1}) \le \dim(\mathcal{S}_{Q}) \le \dim(\mathcal{S}) \le c.
\end{eqnarray*}
If $\rank(\ms\ma) = n$, then we additionally have
\begin{eqnarray*}
	\dim(\mathcal{S}_{1}) \le n \le \dim(\mathcal{S}_{Q}) \le \dim(\mathcal{S}) \le c.
\end{eqnarray*}

\subsection{Interpretation of $\|(\ms\mq)\Dag\ms\mq_{\perp}\|_{2}$ if $\ms$ has orthonormal rows}
\label{sec:angle1}

If $\ms$ has orthonormal rows, the quantity $\|(\ms\mq)\Dag\ms\mq_{\perp}\|_{2}$ has geometric interpretation even with no additional requirements on $\ms$ or $\rank(\ms\ma)$.  One example is sketching via random sampling without replacement where one row is selected in each sample.  The following relies on a key result on the angles between subspaces from \cite[Theorem 3.1]{ZKny13}.

\begin{prop}
	\label{prop:sqdagsqperp1}
	For the perturbed MMLR problem in \eqref{eqn:perturbedMMLR} with the subspaces defined in Section \ref{sec:decomp}, if $\ms$ has orthornomal rows, then
	\begin{eqnarray*}
		\| (\ms\mq)\Dag (\ms\mq_{\perp}) \|_{2} = 
		\tan \theta_{1}(\mathcal{S}, \mathcal{Q}),
	\end{eqnarray*}
	where 
	$\theta_{1}(\mathcal{S},\mathcal{Q})$ denotes a largest principal angle between $\mathcal{S}$ and $\mathcal{Q}$. 
	The absolute error bound in a Schatten $p$-norm is
	\begin{eqnarray*}
		\| \mxtilde - \mxhat \|_{(p)} &\le& \tan \theta_{1}(\mathcal{S}, \mathcal{Q}) \, \|\ma\Dag\|_{2} \, \|\mrhat\|_{(p)}.
	\end{eqnarray*}
\end{prop}

This result follows from \cite[Theorem 3.1]{ZKny13} using the orthogonal matrix $\begin{pmatrix} \mq & \mq_{\perp} \end{pmatrix}$ and $\ms\Tra$ with $\ms$ having orthonormal rows.  Thus, the positive singular values of $(\ms\mq)\Dag\ms\mq_{\perp}$ are the tangents of the principal angles between $\mathcal{S}$ and $\mathcal{Q}$.  Therefore, the absolute error in a Schatten $p$-norm between the sketched and exact MMLR solutions depends on the tangent of a largest principal angle between $\mathcal{S}$ and $\mathcal{Q}$.   Notice that without additional assumptions on $\rank(\ms\ma)$, the tangent of a principal angle between $\mathcal{S}$ and $\mathcal{Q}$ may be $\infty$.

\subsection{Interpretation of $\|(\ms\mq)\Dag\ms\mq_{\perp}\|_{2}$ if $\rank(\ms\ma) = \rank(\ma)$}
\label{sec:angle2}

If the sketching matrix $\ms$ preserves rank so that $\rank(\ms\ma) = \rank(\ma)$, the quantity $\|(\ms\mq)\Dag\ms\mq_{\perp}\|_{2}$ has geometric interpretation without requiring additional assumptions on $\ms$.  This interpretation is based on \cite[Theorem 3.1 and Remark 3.1]{ZKny13}.

\begin{prop}
	\label{prop:sqdagsqperp2}
	For the perturbed MMLR problem in \eqref{eqn:perturbedMMLR} with the subspaces defined in Section \ref{sec:decomp}, if $\rank(\ms\ma) = \rank(\ma)$, then 
	the singular values of $(\ms\mq)\Dag\ms\mq_{\perp}$ represent the tangents of the principal angles between $\mathcal{Z}$, a subspace of $\mathcal{S}_{Q}$, and $\mathcal{Q}$.  Therefore,
	\begin{eqnarray*}
		\| (\ms\mq)\Dag (\ms\mq_{\perp}) \|_{2} = 
		\tan \theta_{1}(\mathcal{Z}, \mathcal{Q}),
	\end{eqnarray*}
	where $\theta_{1}(\mathcal{Z},\mathcal{Q})$ denotes a largest principal angle between $\mathcal{Z}$ and $\mathcal{Q}$.  Moreover, $\tan \theta_{1}(\mathcal{Z},\mathcal{Q})$ is strictly less than $\infty$ and the absolute error bound in a Schatten $p$-norm is
	\begin{eqnarray*}
		\| \mxtilde - \mxhat \|_{(p)} &\le& \tan \theta_{1}(\mathcal{Z}, \mathcal{Q}) \, \|\ma\Dag\|_{2} \, \|\mrhat\|_{(p)}.
	\end{eqnarray*}
\end{prop}

\begin{proof}
	The proof is adapted from \cite[Remark 3.1]{ZKny13}.  
	The proof strategy is to construct an orthonormal basis for a subspace of $\mathcal{S}_{Q}$ and then to apply \cite[Theorem 3.1]{ZKny13} with the orthonormal basis and $\mq$.
	
	We begin with a basis transformation of $\ms$ by constructing the orthogonal matrix
	\begin{eqnarray*}
		\mq_{B} \equiv \begin{pmatrix}
				\mq & \mq_{\perp}
		\end{pmatrix} \in \mathbb{R}^{m \times m}.
	\end{eqnarray*}
	Rewriting $\ms$ in terms of $\mq_{B}$ gives
	\begin{eqnarray*}
		\ms = \ms\mq_{B} \mq_{B}\Tra = \begin{pmatrix} \ms\mq &\ms\mq_{\perp} \end{pmatrix}\mq_{B}\Tra.
	\end{eqnarray*}
	Since $\rank(\ms\mq) = n$, $(\ms\mq)\Dag$ is a left inverse of $\ms\mq$ and so applying it to $\ms$ on the left gives
	\begin{eqnarray*}
		\mz \equiv (\ms\mq)\Dag \ms = \begin{pmatrix} \mi_{n} & (\ms\mq)\Dag\ms\mq_{\perp} \end{pmatrix}\mq_{B}\Tra \in \mathbb{R}^{n \times m}.
	\end{eqnarray*}
	Let $\mt \equiv (\ms\mq)\Dag\ms\mq_{\perp} \in \mathbb{R}^{n \times (m-n)}$.  We will show that the singular values of $\mt$ represent the tangents of the principal angles between $\mathcal{S}$ and $\mathcal{Q}$.
	
	Notice that the Gram matrix
	\begin{eqnarray*}
		\mz\mz\Tra = \mi_{n} + \mt\mt\Tra \in \mathbb{R}^{n \times n}
	\end{eqnarray*}
	is symmetric positive definite.  Therefore, its inverse has the unique symmetric positive definite square root $(\mz\mz\Tra)^{-\frac{1}{2}} = (\mi_{n} + \mt\mt\Tra)^{-\frac{1}{2}}$.  
	Now define
	\begin{eqnarray*}
		\mz_{0} \equiv (\mz\mz\Tra)^{-\frac{1}{2}}\mz \in \mathbb{R}^{n \times m}.
	\end{eqnarray*}
	Then $\mz_{0}$ has orthonormal rows and the columns of $\mz_{0}\Tra$ represent a basis for $\range(\mz\Tra)$.  Since $\rank(\ms\mq) = n$, $\range(\mz\Tra) = \range(\ms\Tra \ms\mq) \subseteq \range(\ms\Tra) = \mathcal{S}$.  
	
	Applying \cite[Theorem 3.1]{ZKny13} with $\mz_{0}\Tra$ and $\mq$ shows that the singular values of $(\ms\mq)\Dag\ms\mq_{\perp}$ are the tangents of the principal angles between $\mathcal{Z} \equiv \range(\mz\Tra)$ and $\mathcal{Q}$.  
	Since $(\mz_{0}\Tra)\Tra\mq = \mz_{0}\mq = (\mz\mz)^{-\frac{1}{2}} = (\mi_{n} + \mt\mt\Tra)^{-\frac{1}{2}}$ is nonsingular,  
	$\mathcal{Z} \subseteq \mathcal{S}_{Q}$ and the tangents of the principal angles between $\mathcal{Z}$ and $\mathcal{Q}$ are strictly less than $\infty$.	
\end{proof}

	Clearly, $\mathcal{Z} \subseteq \mathcal{S}_{Q}$.  One might ask the question: Is $\mathcal{Z} = \mathcal{S}_{Q}$?  
	Notice that $\rank(\ms\ma) = n$ and $\rank(\ms\Tra) \le c$ imply that
	\begin{eqnarray*}
		n \le \dim(\mathcal{S}_{Q}) \le c \quad \text{ and } \quad \dim(\mathcal{S}_{1}) \le c - n.
	\end{eqnarray*}
	Although $\mathcal{Z} \ne \mathcal{S}_{Q}$ in general, if $\dim(\mathcal{S}_{Q}) = n$, then $\dim(\mathcal{Z}) = n$ implies that $\mathcal{Z} = \mathcal{S}_{Q}$.  Meanwhile, if $\dim(\mathcal{S}_{Q}) > n$, then $n = \dim(\mathcal{Z}) < \dim(\mathcal{S}_{Q})$ so that $\mathcal{Z} \ne \mathcal{S}_{Q}$.   
	The example in Section \ref{sec:example} illustrates this concretely.

Propositions \ref{prop:sqdagsqperp1} and \ref{prop:sqdagsqperp2} show that if $\rank(\ms\ma) = \rank(\ma)$, $\| (\ms\mq)\Dag \ms\mq_{\perp} \|_{2}$ has geometric interpretation as the tangent of a largest principal angle between a subspace of $\mathcal{S}_{Q}$ and $\mathcal{Q}$.  Moreover, the tangents of the principal angles between these two subspaces are bounded.  If $\rank(\ms\ma) < \rank(\ma)$, then $\| (\ms\mq)\Dag \ms\mq_{\perp} \|_{2}$ still has geometric interpretation as the tangent of a largest principal angle between $\mathcal{S}$ and $\mathcal{Q}$ if $\ms$ has orthonormal rows.  
Proposition \ref{prop:sqdagsqperp2} implies that if $\rank(\ms\ma) = \rank(\ma)$, then the operator norm difference between $\mP$ and $\mPa$ has the following geometric interpretation.

\begin{prop}\label{prop:dist_Pa-P}
	For the perturbed MMLR problem in \eqref{eqn:perturbedMMLR} with the subspaces defined in Section \ref{sec:decomp}, if $\rank(\ms\ma) = \rank(\ma)$, 
	\begin{eqnarray*}
		\| \mP - \mPa \|_{2} &=& \tan \theta_{1}(\mathcal{Z}, \mathcal{Q}),
	\end{eqnarray*}
	where $\mathcal{Z}$ is a subspace of $\mathcal{S}_{Q}$ and $\theta_{1}(\mathcal{Z},\mathcal{Q})$ denotes a largest principal angle between $\mathcal{Z}$ and $\mathcal{Q}$.  Moreover, $\tan \theta_{1}(\mathcal{Z}, \mathcal{Q})$ is strictly less than $\infty$.
\end{prop}

\begin{proof}
	We decompose $\mi_{m}$ into the sum of orthogonal projectors and rewrite the operator norm difference between $\mPa$ and $\mP$ as the following 
	\begin{eqnarray*}
		\mPa - \mP = \mq\mq\Tra - \mq(\ms\mq)\Dag\ms(\mq\mq\Tra + \mq_{\perp}\mq_{\perp}\Tra).
	\end{eqnarray*}
	After we expand and cancel terms, the result follows from unitary invariance of spectral norms and Proposition \ref{prop:sqdagsqperp2}.
\end{proof}

This result is implied from the absolute error bound in Proposition \ref{prop:sqdagsqperp2}.  However, the direct statement of this result ties the interpretation of $\|(\ms\mq)\Dag\ms\mq_{\perp}\|_{2}$ as the tangent of a largest principal angle between a subspace of $\mathcal{S}_{Q}$ and $\mathcal{Q}$ to the operator norm difference between $\mP$ and $\mPa$.  In this way, we have additional geometric interpretation of the difference between an orthogonal and oblique projector with the same range if $\ms$ preserves rank.

\subsubsection{Illustrative example of the subspaces in Proposition \ref{prop:sqdagsqperp2}}\label{sec:example}

We provide an example illustrating the subspaces of Section \ref{sec:decomp} in the context of Proposition \ref{prop:sqdagsqperp2}.  Let 
\begin{eqnarray*}
	\mq = \begin{pmatrix}
		1 & 0 & 0 \\
		0 & 1 & 0 \\
		0 & 0 & 1 \\
		0 & 0 & 0 \\
		0 & 0 & 0 \\
		0 & 0 & 0 \end{pmatrix}, %= \range\begin{pmatrix} \ve_{1} & \ve_{2} & \ve_{3} \end{pmatrix}
	 \;
	\mq_{\perp} = \begin{pmatrix}
		0 & 0 & 0 \\
		0 & 0 & 0 \\
		0 & 0 & 0 \\
		1 & 0 & 0 \\
		0 & 1 & 0 \\
		0 & 0 & 1 \end{pmatrix}, %= \range\begin{pmatrix} \ve_{4} & \ve_{5} & \ve_{6} \end{pmatrix},
%\end{eqnarray*}
\; \text{ and } \;
%\begin{eqnarray*}
	\ms\Tra = \begin{pmatrix}
		1 & 0 & 0 & 0 \\
		0 & 1 & 0 & 1 \\
		0 & 0 & 1 & 0 \\
		0 & 0 & 0 & 0 \\
		0 & 0 & 0 & 1 \\
		0 & 0 & 1 & 0 \end{pmatrix}. %= \begin{pmatrix} \ve_{1} & \ve_{2} & \ve_{3} + \ve_{6}  \end{pmatrix}.
\end{eqnarray*}
Then $\mathcal{S}$ has the following subspaces
\begin{eqnarray*}
	\mathcal{S}_{1} = \range\begin{pmatrix}
		1 & 0  \\
		0 & 1  \\
		0 & 0  \\
		0 & 0  \\
		0 & 0  \\
		0 & 0  \end{pmatrix},
	\;
	\mathcal{S}_{10} = \range\begin{pmatrix}
		0 & 0 \\
		0 & 1 \\
		1 & 0 \\
		0 & 0 \\
		0 & 1 \\
		1 & 0 \end{pmatrix}, 
	\; \text{ and } \;
		\mathcal{S}_{Q} = \range\begin{pmatrix}
		1 & 0 & 0 & 0 \\
		0 & 1 & 0 & 1 \\
		0 & 0 & 1 & 0 \\
		0 & 0 & 0 & 0 \\
		0 & 0 & 0 & 1 \\
		0 & 0 & 1 & 0 \end{pmatrix}.
\end{eqnarray*}
This example illustrates how $\mathcal{S}_{1}$ contains directions in $\mathcal{S}$ that are in $\mathcal{Q}$, and $\mathcal{S}_{10}$ contains directions in $\mathcal{S}$ that cannot be represented solely by directions in $\mathcal{Q}$ or directions in $\mathcal{Q}^{\perp}$.  This is because vectors in $\mathcal{S}_{10}$ are obtained from a non-trivial linear combination of vectors in $\mathcal{Q}$ with vectors in $\mathcal{Q}^{\perp}$.  Thus, for any $\vv \in \mathcal{S}_{10}$ and any $\vq \in \mathcal{Q}$, we have $\vv\Tra\vq \ne 0$.  However, $\vv \notin \mathcal{Q}$ and $\vv \notin \mathcal{Q}^{\perp}$.

Notice that in this example, there are no non-zero directions in $\mathcal{S}$ that are also in $\mathcal{Q}^{\perp}$.  Since $\rank(\ms\ma) = n$ and $\rank(\ms\Tra) \le c$ require that $\dim(\mathcal{S}_{Q}) \ge n$ and $\dim(\mathcal{S}_{1}) \le c - n$, $\mathcal{S}_{0} = \{\textbf{0}\}$ is an artifact of this example.
 
Proceeding with the example, we have
\begin{eqnarray*}
	\ms\mq = \begin{pmatrix}
		1 & 0 & 0 \\
		0 & 1 & 0 \\
		0 & 0 & 1 \\
		0 & 1 & 0
	\end{pmatrix}
%\end{eqnarray*}
	\; \text{ and } \;
%\begin{eqnarray*}
	\mz = (\ms\mq)\Dag\ms = \begin{pmatrix}
		1 & 0 & 0 & 0 & 0 & 0 \\
		0 & 1 & 0 & 0 & \frac{1}{2} & 0 \\
		0 & 0 & 1 & 0 & 0 & 1
	\end{pmatrix},
\end{eqnarray*}
where $\ms\mq$ has full column rank.  This gives us
\begin{eqnarray*}
	\mz\mz\Tra = \begin{pmatrix}
		1 & 0 & 0 \\
		0 & \frac{5}{4} & 0 \\
		0 & 0 & 2
	\end{pmatrix}
%\end{eqnarray*}
	\; \text{ and } \;
%\begin{eqnarray*}
	\mz_{0} = (\mz\mz\Tra)^{-\frac{1}{2}}\mz = \begin{pmatrix}
	1 & 0 & 0 & 0 & 0 & 0 \\
	0 & \frac{2\sqrt{5}}{5} & 0 & 0 & \frac{\sqrt{5}}{5} & 0 \\
	0 & 0 & \frac{\sqrt{2}}{2} & 0 & 0 & \frac{\sqrt{2}}{2}
	\end{pmatrix}.
\end{eqnarray*}
Thus, $\mz_{0}\Tra$ has orthonormal columns and
\begin{eqnarray*}
	\mz_{0}\mq = \begin{pmatrix}
		1 & 0 & 0 \\
		0 & \frac{2\sqrt{5}}{5} & 0 \\
		0 & 0 & \frac{\sqrt{2}}{2} 
	\end{pmatrix}
\end{eqnarray*}
is nonsingular so that $\mathcal{Z} \subseteq \mathcal{S}_{Q}$ since all three directions in $\mathcal{Z}$ are not orthogonal with directions in $\mathcal{Q}$.  However, $\dim(\mathcal{Z}) = 3 = n$ while $\dim(\mathcal{S}_{Q}) = 4 = c$ so that $\mathcal{Z} \ne \mathcal{S}_{Q}$.

\noindent{\bf Funding:} The work was supported in part by NSF grants DGE-1633587, DMS-1760374, and DMS-1745654.

\bibliography{LSBV}

\begin{thebibliography}{60}
\expandafter\ifx\csname natexlab\endcsname\relax\def\natexlab#1{#1}\fi
\providecommand{\url}[1]{\texttt{#1}}
\providecommand{\href}[2]{#2}
\providecommand{\path}[1]{#1}
\providecommand{\DOIprefix}{doi:}
\providecommand{\ArXivprefix}{arXiv:}
\providecommand{\URLprefix}{URL: }
\providecommand{\Pubmedprefix}{pmid:}
\providecommand{\doi}[1]{\href{http://dx.doi.org/#1}{\path{#1}}}
\providecommand{\Pubmed}[1]{\href{pmid:#1}{\path{#1}}}
\providecommand{\bibinfo}[2]{#2}
\ifx\xfnm\relax \def\xfnm[#1]{\unskip,\space#1}\fi
%Type = Article
\bibitem[{Ailon and Chazelle(2009)}]{Ailon2009}
\bibinfo{author}{Ailon, N.}, \bibinfo{author}{Chazelle, B.},
  \bibinfo{year}{2009}.
\newblock \bibinfo{title}{The fast {Johnson} {Lindenstrauss} transform and
  approximate nearest neighbors}.
\newblock \bibinfo{journal}{SIAM Journal on Scientific Computing}
  \bibinfo{volume}{39}, \bibinfo{pages}{302--322}.
%Type = Article
\bibitem[{Avron et~al.(2010)Avron, Maymounkov and Toledo}]{avron2010blendenpik}
\bibinfo{author}{Avron, H.}, \bibinfo{author}{Maymounkov, P.},
  \bibinfo{author}{Toledo, S.}, \bibinfo{year}{2010}.
\newblock \bibinfo{title}{Blendenpik: Supercharging {LAPACK}'s least-squares
  solver}.
\newblock \bibinfo{journal}{SIAM Journal on Scientific Computing}
  \bibinfo{volume}{32}, \bibinfo{pages}{1217--1236}.
%Type = Article
\bibitem[{Boutsidis and Drineas(2009)}]{BD09}
\bibinfo{author}{Boutsidis, C.}, \bibinfo{author}{Drineas, P.},
  \bibinfo{year}{2009}.
\newblock \bibinfo{title}{Random projections for the nonnegative least-squares
  problem}.
\newblock \bibinfo{journal}{Linear Algebra and its Applications}
  \bibinfo{volume}{431}, \bibinfo{pages}{760--771}.
%Type = Article
\bibitem[{Breiman and Friedman(1997)}]{breiman1997predicting}
\bibinfo{author}{Breiman, L.}, \bibinfo{author}{Friedman, J.H.},
  \bibinfo{year}{1997}.
\newblock \bibinfo{title}{Predicting multivariate responses in multiple linear
  regression}.
\newblock \bibinfo{journal}{Journal of the Royal Statistical Society: Series B
  (Statistical Methodology)} \bibinfo{volume}{59}, \bibinfo{pages}{3--54}.
%Type = Article
\bibitem[{Brust et~al.(2020)Brust, Marcia and Petra}]{brust2020computationally}
\bibinfo{author}{Brust, J.J.}, \bibinfo{author}{Marcia, R.F.},
  \bibinfo{author}{Petra, C.G.}, \bibinfo{year}{2020}.
\newblock \bibinfo{title}{Computationally efficient decompositions of oblique
  projection matrices}.
\newblock \bibinfo{journal}{SIAM Journal on Matrix Analysis and Applications}
  \bibinfo{volume}{41}, \bibinfo{pages}{852--870}.
%Type = Article
\bibitem[{Candes and Plan(2010)}]{candes2010matrix}
\bibinfo{author}{Candes, E.J.}, \bibinfo{author}{Plan, Y.},
  \bibinfo{year}{2010}.
\newblock \bibinfo{title}{Matrix completion with noise}.
\newblock \bibinfo{journal}{Proceedings of the IEEE} \bibinfo{volume}{98},
  \bibinfo{pages}{925--936}.
%Type = Article
\bibitem[{Cand{\`e}s and Recht(2009)}]{candes2009exact}
\bibinfo{author}{Cand{\`e}s, E.J.}, \bibinfo{author}{Recht, B.},
  \bibinfo{year}{2009}.
\newblock \bibinfo{title}{Exact matrix completion via convex optimization}.
\newblock \bibinfo{journal}{Foundations of Computational Mathematics}
  \bibinfo{volume}{9}, \bibinfo{pages}{717}.
%Type = Article
\bibitem[{Chatterjee and Hadi(1986)}]{ChatterH86}
\bibinfo{author}{Chatterjee, S.}, \bibinfo{author}{Hadi, A.S.},
  \bibinfo{year}{1986}.
\newblock \bibinfo{title}{Influential observations, high leverage points, and
  outliers in linear regression}.
\newblock \bibinfo{journal}{Statistical Science} \bibinfo{volume}{1},
  \bibinfo{pages}{379--416}.
\newblock \bibinfo{note}{With discussion}.
%Type = Article
\bibitem[{Chi and Ipsen(2019)}]{chi2018randomized}
\bibinfo{author}{Chi, J.T.}, \bibinfo{author}{Ipsen, I.C.F.},
  \bibinfo{year}{2019}.
\newblock \bibinfo{title}{A geometric analysis of model- and algorithm-induced
  uncertainties for randomized least squares regression}.
\newblock \bibinfo{journal}{arXiv preprint:1808.05924} .
%Type = Article
\bibitem[{Chmieli{\'n}ski(2005)}]{chmielinski2005linear}
\bibinfo{author}{Chmieli{\'n}ski, J.}, \bibinfo{year}{2005}.
\newblock \bibinfo{title}{Linear mappings approximately preserving
  orthogonality}.
\newblock \bibinfo{journal}{Journal of mathematical analysis and applications}
  \bibinfo{volume}{304}, \bibinfo{pages}{158--169}.
%Type = Inproceedings
\bibitem[{Drineas et~al.(2006)Drineas, Mahoney and
  Muthukrishnan}]{DrineasMM2006}
\bibinfo{author}{Drineas, P.}, \bibinfo{author}{Mahoney, M.W.},
  \bibinfo{author}{Muthukrishnan, S.}, \bibinfo{year}{2006}.
\newblock \bibinfo{title}{Sampling algorithms for {$l_2$} regression and
  applications}, in: \bibinfo{booktitle}{Proceedings of the 17th {A}nnual
  {ACM}-{SIAM} {S}ymposium on {D}iscrete {A}lgorithms},
  \bibinfo{publisher}{ACM, New York}. pp. \bibinfo{pages}{1127--1136}.
%Type = Article
\bibitem[{Drineas et~al.(2011)Drineas, Mahoney, Muthukrishnan and
  Sarl{\'{o}}s}]{Drineas2011}
\bibinfo{author}{Drineas, P.}, \bibinfo{author}{Mahoney, M.W.},
  \bibinfo{author}{Muthukrishnan, S.}, \bibinfo{author}{Sarl{\'{o}}s, T.},
  \bibinfo{year}{2011}.
\newblock \bibinfo{title}{{Faster least squares approximation}}.
\newblock \bibinfo{journal}{Numerische Mathematik} \bibinfo{volume}{117},
  \bibinfo{pages}{219--249}.
%Type = Article
\bibitem[{Drma\v{c} and Saibaba(2018)}]{MR3826673}
\bibinfo{author}{Drma\v{c}, Z.}, \bibinfo{author}{Saibaba, A.K.},
  \bibinfo{year}{2018}.
\newblock \bibinfo{title}{The discrete empirical interpolation method:
  canonical structure and formulation in weighted inner product spaces}.
\newblock \bibinfo{journal}{SIAM Journal on Matrix Analysis and Applications}
  \bibinfo{volume}{39}, \bibinfo{pages}{1152--1180}.
\newblock \DOIprefix\doi{10.1137/17M1129635}.
%Type = Article
\bibitem[{Du et~al.(2017)Du, Zhao, Wang and Hu}]{du2017two}
\bibinfo{author}{Du, H.}, \bibinfo{author}{Zhao, Z.}, \bibinfo{author}{Wang,
  S.}, \bibinfo{author}{Hu, Q.}, \bibinfo{year}{2017}.
\newblock \bibinfo{title}{Two-dimensional discriminant analysis based on
  {S}chatten p-norm for image feature extraction}.
\newblock \bibinfo{journal}{Journal of Visual Communication and Image
  Representation} \bibinfo{volume}{45}, \bibinfo{pages}{87--94}.
%Type = Article
\bibitem[{Eyvazian et~al.(2011)Eyvazian, Noorossana, Saghaei and
  Amiri}]{eyvazian2011phase}
\bibinfo{author}{Eyvazian, M.}, \bibinfo{author}{Noorossana, R.},
  \bibinfo{author}{Saghaei, A.}, \bibinfo{author}{Amiri, A.},
  \bibinfo{year}{2011}.
\newblock \bibinfo{title}{Phase {II} monitoring of multivariate multiple linear
  regression profiles}.
\newblock \bibinfo{journal}{Quality and Reliability Engineering International}
  \bibinfo{volume}{27}, \bibinfo{pages}{281--296}.
%Type = Article
\bibitem[{Han et~al.(2017)Han, Malioutov, Avron and
  Shin}]{han2017approximating}
\bibinfo{author}{Han, I.}, \bibinfo{author}{Malioutov, D.},
  \bibinfo{author}{Avron, H.}, \bibinfo{author}{Shin, J.},
  \bibinfo{year}{2017}.
\newblock \bibinfo{title}{Approximating spectral sums of large-scale matrices
  using stochastic {C}hebyshev approximations}.
\newblock \bibinfo{journal}{SIAM Journal on Scientific Computing}
  \bibinfo{volume}{39}, \bibinfo{pages}{A1558--A1585}.
%Type = Article
\bibitem[{Hansen(2013)}]{MR3021435}
\bibinfo{author}{Hansen, P.C.}, \bibinfo{year}{2013}.
\newblock \bibinfo{title}{Oblique projections and standard-form transformations
  for discrete inverse problems}.
\newblock \bibinfo{journal}{Numerical Linear Algebra with Applications}
  \bibinfo{volume}{20}, \bibinfo{pages}{250--258}.
\newblock \DOIprefix\doi{10.1002/nla.802}.
%Type = Article
\bibitem[{Hn\v{e}tynkov\'{a} et~al.(2011)Hn\v{e}tynkov\'{a}, Ple\v{s}inger,
  Sima, Strako\v{s} and Van~Huffel}]{tls}
\bibinfo{author}{Hn\v{e}tynkov\'{a}, I.}, \bibinfo{author}{Ple\v{s}inger, M.},
  \bibinfo{author}{Sima, D.M.}, \bibinfo{author}{Strako\v{s}, Z.},
  \bibinfo{author}{Van~Huffel, S.}, \bibinfo{year}{2011}.
\newblock \bibinfo{title}{The total least squares problem in {$AX\approx B$}: a
  new classification with the relationship to the classical works}.
\newblock \bibinfo{journal}{SIAM Journal on Matrix Analysis and Applications}
  \bibinfo{volume}{32}, \bibinfo{pages}{748--770}.
%Type = Article
\bibitem[{Hn\v{e}tynkov\'{a} et~al.(2013)Hn\v{e}tynkov\'{a}, Ple\v{s}inger and
  Strako\v{s}}]{core_lsmultRHS}
\bibinfo{author}{Hn\v{e}tynkov\'{a}, I.}, \bibinfo{author}{Ple\v{s}inger, M.},
  \bibinfo{author}{Strako\v{s}, Z.}, \bibinfo{year}{2013}.
\newblock \bibinfo{title}{The core problem within a linear approximation
  problem {$AX\approx B$} with multiple right-hand sides}.
\newblock \bibinfo{journal}{SIAM Journal on Matrix Analysis and Applications}
  \bibinfo{volume}{34}, \bibinfo{pages}{917--931}.
%Type = Article
\bibitem[{Hoaglin and Welsch(1978)}]{HoagW78}
\bibinfo{author}{Hoaglin, D.C.}, \bibinfo{author}{Welsch, R.E.},
  \bibinfo{year}{1978}.
\newblock \bibinfo{title}{The {Hat} matrix in regression and {ANOVA}}.
\newblock \bibinfo{journal}{The American Statistician} \bibinfo{volume}{32},
  \bibinfo{pages}{17--22}.
%Type = Article
\bibitem[{Ipsen and Wentworth(2014)}]{ipsen2014effect}
\bibinfo{author}{Ipsen, I.C.}, \bibinfo{author}{Wentworth, T.},
  \bibinfo{year}{2014}.
\newblock \bibinfo{title}{The effect of coherence on sampling from matrices
  with orthonormal columns, and preconditioned least squares problems}.
\newblock \bibinfo{journal}{SIAM Journal on Matrix Analysis and Applications}
  \bibinfo{volume}{35}, \bibinfo{pages}{1490--1520}.
%Type = Article
\bibitem[{Jeong et~al.(2012)Jeong, St-Hilaire, Ouarda and
  Gachon}]{jeong2012multisite}
\bibinfo{author}{Jeong, D.I.}, \bibinfo{author}{St-Hilaire, A.},
  \bibinfo{author}{Ouarda, T.B.}, \bibinfo{author}{Gachon, P.},
  \bibinfo{year}{2012}.
\newblock \bibinfo{title}{Multisite statistical downscaling model for daily
  precipitation combined by multivariate multiple linear regression and
  stochastic weather generator}.
\newblock \bibinfo{journal}{Climatic Change} \bibinfo{volume}{114},
  \bibinfo{pages}{567--591}.
%Type = Book
\bibitem[{Johnson and Horn(1985)}]{johnson1985matrix}
\bibinfo{author}{Johnson, C.R.}, \bibinfo{author}{Horn, R.A.},
  \bibinfo{year}{1985}.
\newblock \bibinfo{title}{Topics in Matrix Analysis}.
\newblock \bibinfo{publisher}{Cambridge {U}niversity {P}ress}.
%Type = Article
\bibitem[{Kitahara and Tsuchiya(2009)}]{MR2558818}
\bibinfo{author}{Kitahara, T.}, \bibinfo{author}{Tsuchiya, T.},
  \bibinfo{year}{2009}.
\newblock \bibinfo{title}{Proximity of weighted and layered least squares
  solutions}.
\newblock \bibinfo{journal}{SIAM Journal on Matrix Analysis and Applications}
  \bibinfo{volume}{31}, \bibinfo{pages}{1172--1186}.
\newblock \DOIprefix\doi{10.1137/080725787}.
%Type = Misc
\bibitem[{Larsen and Kolda(2020)}]{larsen2020practical}
\bibinfo{author}{Larsen, B.W.}, \bibinfo{author}{Kolda, T.G.},
  \bibinfo{year}{2020}.
\newblock \bibinfo{title}{Practical leverage-based sampling for low-rank tensor
  decomposition}.
\newblock \bibinfo{howpublished}{arXiv preprint arXiv:2006.16438}.
%Type = Article
\bibitem[{Lefkimmiatis et~al.(2013)Lefkimmiatis, Ward and
  Unser}]{lefkimmiatis2013hessian}
\bibinfo{author}{Lefkimmiatis, S.}, \bibinfo{author}{Ward, J.P.},
  \bibinfo{author}{Unser, M.}, \bibinfo{year}{2013}.
\newblock \bibinfo{title}{Hessian {S}chatten-norm regularization for linear
  inverse problems}.
\newblock \bibinfo{journal}{IEEE Transactions on Image Processing}
  \bibinfo{volume}{22}, \bibinfo{pages}{1873--1888}.
%Type = Article
\bibitem[{Li et~al.(2015)Li, Nan and Zhu}]{li2015multivariate}
\bibinfo{author}{Li, Y.}, \bibinfo{author}{Nan, B.}, \bibinfo{author}{Zhu, J.},
  \bibinfo{year}{2015}.
\newblock \bibinfo{title}{Multivariate sparse group {L}asso for the
  multivariate multiple linear regression with an arbitrary group structure}.
\newblock \bibinfo{journal}{Biometrics} \bibinfo{volume}{71},
  \bibinfo{pages}{354--363}.
%Type = Inproceedings
\bibitem[{Luo et~al.(2014)Luo, Yang, Chen and Gao}]{luo2014schatten}
\bibinfo{author}{Luo, L.}, \bibinfo{author}{Yang, J.}, \bibinfo{author}{Chen,
  J.}, \bibinfo{author}{Gao, Y.}, \bibinfo{year}{2014}.
\newblock \bibinfo{title}{{S}chatten $p$-norm based matrix regression model for
  image classification}, in: \bibinfo{booktitle}{Chinese Conference on Pattern
  Recognition}, \bibinfo{organization}{Springer}. pp.
  \bibinfo{pages}{140--150}.
%Type = Inproceedings
\bibitem[{Ma et~al.(2014)Ma, Mahoney and Yu}]{MMY14}
\bibinfo{author}{Ma, P.}, \bibinfo{author}{Mahoney, M.W.}, \bibinfo{author}{Yu,
  B.}, \bibinfo{year}{2014}.
\newblock \bibinfo{title}{A statistical perspective on algorithmic leveraging},
  in: \bibinfo{booktitle}{Proceedings of the 31st International Conference on
  International Conference on Machine Learning}, pp.
  \bibinfo{pages}{I--91--I--99}.
%Type = Article
\bibitem[{Ma et~al.(2015)Ma, Mahoney and Yu}]{MMY15}
\bibinfo{author}{Ma, P.}, \bibinfo{author}{Mahoney, M.W.}, \bibinfo{author}{Yu,
  B.}, \bibinfo{year}{2015}.
\newblock \bibinfo{title}{A statistical perspective on algorithmic leveraging}.
\newblock \bibinfo{journal}{Journal of Machine Learning Research}
  \bibinfo{volume}{16}, \bibinfo{pages}{861--911}.
%Type = Article
\bibitem[{Maher(1990)}]{maher1990op}
\bibinfo{author}{Maher, P.J.}, \bibinfo{year}{1990}.
\newblock \bibinfo{title}{Some operator inequalities concerning generalized
  inverses}.
\newblock \bibinfo{journal}{Illinois Journal of Mathematics}
  \bibinfo{volume}{34}, \bibinfo{pages}{503--514}.
%Type = Article
\bibitem[{Maher(1992)}]{maher1992some}
\bibinfo{author}{Maher, P.J.}, \bibinfo{year}{1992}.
\newblock \bibinfo{title}{Some norm inequalities concerning generalized
  inverses}.
\newblock \bibinfo{journal}{Linear Algebra and its Applications}
  \bibinfo{volume}{174}, \bibinfo{pages}{99--110}.
%Type = Article
\bibitem[{Maher(2007a)}]{maher2007some}
\bibinfo{author}{Maher, P.J.}, \bibinfo{year}{2007}a.
\newblock \bibinfo{title}{Some norm inequalities concerning generalized
  inverses, 2}.
\newblock \bibinfo{journal}{Linear Algebra and its Applications}
  \bibinfo{volume}{420}, \bibinfo{pages}{517--525}.
%Type = Article
\bibitem[{Maher(2007b)}]{maher2007b}
\bibinfo{author}{Maher, P.J.}, \bibinfo{year}{2007}b.
\newblock \bibinfo{title}{Some singular values, and unitarily invariant norm
  inequalities concerning generalized inverses}.
\newblock \bibinfo{journal}{Filomat} \bibinfo{volume}{21},
  \bibinfo{pages}{99--111}.
%Type = Article
\bibitem[{Mahoney(2011)}]{mahoney2011randomized}
\bibinfo{author}{Mahoney, M.W.}, \bibinfo{year}{2011}.
\newblock \bibinfo{title}{Randomized algorithms for matrices and data}.
\newblock \bibinfo{journal}{Foundations and Trends{\textregistered} in Machine
  Learning} \bibinfo{volume}{3}, \bibinfo{pages}{123--224}.
%Type = Article
\bibitem[{Meng et~al.(2014)Meng, Saunders and Mahoney}]{LSRN}
\bibinfo{author}{Meng, X.}, \bibinfo{author}{Saunders, M.A.},
  \bibinfo{author}{Mahoney, M.W.}, \bibinfo{year}{2014}.
\newblock \bibinfo{title}{L{SRN}: a parallel iterative solver for strongly
  over- or underdetermined systems}.
\newblock \bibinfo{journal}{SIAM Journal on Scientific Computing}
  \bibinfo{volume}{36}, \bibinfo{pages}{C95--C118}.
%Type = Article
\bibitem[{Moj{\v{s}}kerc and Turn{\v{s}}ek(2010)}]{mojvskerc2010mappings}
\bibinfo{author}{Moj{\v{s}}kerc, B.}, \bibinfo{author}{Turn{\v{s}}ek, A.},
  \bibinfo{year}{2010}.
\newblock \bibinfo{title}{Mappings approximately preserving orthogonality in
  normed spaces}.
\newblock \bibinfo{journal}{Nonlinear Analysis: Theory, Methods \&
  Applications} \bibinfo{volume}{73}, \bibinfo{pages}{3821--3831}.
%Type = Article
\bibitem[{Noorossana et~al.(2010)Noorossana, Eyvazian, Amiri and
  Mahmoud}]{noorossana2010statistical}
\bibinfo{author}{Noorossana, R.}, \bibinfo{author}{Eyvazian, M.},
  \bibinfo{author}{Amiri, A.}, \bibinfo{author}{Mahmoud, M.A.},
  \bibinfo{year}{2010}.
\newblock \bibinfo{title}{Statistical monitoring of multivariate multiple
  linear regression profiles in phase {I} with calibration application}.
\newblock \bibinfo{journal}{Quality and Reliability Engineering International}
  \bibinfo{volume}{26}, \bibinfo{pages}{291--303}.
%Type = Article
\bibitem[{Raskutti and Mahoney(2016)}]{RM16}
\bibinfo{author}{Raskutti, G.}, \bibinfo{author}{Mahoney, M.W.},
  \bibinfo{year}{2016}.
\newblock \bibinfo{title}{A statistical perspective on randomized sketching for
  ordinary least-squares}.
\newblock \bibinfo{journal}{Journal of Machine Learning Research}
  \bibinfo{volume}{17}, \bibinfo{pages}{Paper No. 214, 31}.
%Type = Article
\bibitem[{Rokhlin and Tygert(2008)}]{RT08}
\bibinfo{author}{Rokhlin, V.}, \bibinfo{author}{Tygert, M.},
  \bibinfo{year}{2008}.
\newblock \bibinfo{title}{A fast randomized algorithm for overdetermined linear
  least-squares regression}.
\newblock \bibinfo{journal}{Proceedings of the National Academy of Sciences
  USA} \bibinfo{volume}{105}, \bibinfo{pages}{13212--13217}.
%Type = Inproceedings
\bibitem[{Sarl\'{o}s(2006)}]{Sarlos2006}
\bibinfo{author}{Sarl\'{o}s, T.}, \bibinfo{year}{2006}.
\newblock \bibinfo{title}{{Improved Approximation Algorithms for Large Matrices
  via Random Projections}}, in: \bibinfo{booktitle}{47th Annual IEEE Symposium
  on Foundations of Computer Science (FOCS'06)}, \bibinfo{publisher}{IEEE}. pp.
  \bibinfo{pages}{143--152}.
%Type = Article
\bibitem[{Stewart(1989)}]{MR976336}
\bibinfo{author}{Stewart, G.W.}, \bibinfo{year}{1989}.
\newblock \bibinfo{title}{On scaled projections and pseudoinverses}.
\newblock \bibinfo{journal}{Linear Algebra and its Applications}
  \bibinfo{volume}{112}, \bibinfo{pages}{189--193}.
\newblock \DOIprefix\doi{10.1016/0024-3795(89)90594-6}.
%Type = Article
\bibitem[{Stewart(2011)}]{Ste2011}
\bibinfo{author}{Stewart, G.W.}, \bibinfo{year}{2011}.
\newblock \bibinfo{title}{On the numerical analysis of oblique projectors}.
\newblock \bibinfo{journal}{SIAM J. Matrix Anal. Appl.} \bibinfo{volume}{32},
  \bibinfo{pages}{309--348}.
%Type = Book
\bibitem[{Stewart and Sun(1990)}]{MR1061154}
\bibinfo{author}{Stewart, G.W.}, \bibinfo{author}{Sun, J.G.},
  \bibinfo{year}{1990}.
\newblock \bibinfo{title}{Matrix perturbation theory}.
\newblock Computer Science and Scientific Computing,
  \bibinfo{publisher}{Academic Press, Inc., Boston, MA}.
%Type = Article
\bibitem[{Sun(1996)}]{optimal_lsmultRHS}
\bibinfo{author}{Sun, J.G.}, \bibinfo{year}{1996}.
\newblock \bibinfo{title}{Optimal backward perturbation bounds for the linear
  least-squares problem with multiple right-hand sides}.
\newblock \bibinfo{journal}{IMA Journal of Numerical Analysis}
  \bibinfo{volume}{16}, \bibinfo{pages}{1--11}.
%Type = Misc
\bibitem[{{Thanei} et~al.(2017){Thanei}, {Heinze} and {Meinshausen}}]{THM2017}
\bibinfo{author}{{Thanei}, G.A.}, \bibinfo{author}{{Heinze}, C.},
  \bibinfo{author}{{Meinshausen}, N.}, \bibinfo{year}{2017}.
\newblock \bibinfo{title}{{Random Projections For Large-Scale Regression}}.
\newblock \href{http://arxiv.org/abs/1701.05325}{{\tt arXiv:1701.05325}}.
%Type = Article
\bibitem[{Turn{\v{s}}ek(2007)}]{turnvsek2007mappings}
\bibinfo{author}{Turn{\v{s}}ek, A.}, \bibinfo{year}{2007}.
\newblock \bibinfo{title}{On mappings approximately preserving orthogonality}.
\newblock \bibinfo{journal}{Journal of mathematical analysis and applications}
  \bibinfo{volume}{336}, \bibinfo{pages}{625--631}.
%Type = Article
\bibitem[{Ubaru et~al.(2017)Ubaru, Chen and Saad}]{ubaru2017fast}
\bibinfo{author}{Ubaru, S.}, \bibinfo{author}{Chen, J.}, \bibinfo{author}{Saad,
  Y.}, \bibinfo{year}{2017}.
\newblock \bibinfo{title}{Fast estimation of {$\text{\tt tr}(f(A))$} via
  stochastic {L}anczos quadrature}.
\newblock \bibinfo{journal}{SIAM Journal on Matrix Analysis and Applications}
  \bibinfo{volume}{38}, \bibinfo{pages}{1075--1099}.
%Type = Article
\bibitem[{\v{C}ern\'{y}(2009)}]{cerny}
\bibinfo{author}{\v{C}ern\'{y}, A.}, \bibinfo{year}{2009}.
\newblock \bibinfo{title}{Characterization of the oblique projector
  {$U(VU)^\dagger V$} with application to constrained least squares}.
\newblock \bibinfo{journal}{Linear Algebra and its Applications}
  \bibinfo{volume}{431}, \bibinfo{pages}{1564--1570}.
\newblock \DOIprefix\doi{10.1016/j.laa.2009.05.025}.
%Type = Article
\bibitem[{Velleman and Welsch(1981)}]{VelleW81}
\bibinfo{author}{Velleman, P.F.}, \bibinfo{author}{Welsch, R.E.},
  \bibinfo{year}{1981}.
\newblock \bibinfo{title}{Efficient computing of regression diagnostics}.
\newblock \bibinfo{journal}{The American Statistician} \bibinfo{volume}{35},
  \bibinfo{pages}{234--242}.
%Type = Article
\bibitem[{Wang et~al.(2018)Wang, Zhu and Ma}]{WZM18}
\bibinfo{author}{Wang, H.}, \bibinfo{author}{Zhu, R.}, \bibinfo{author}{Ma,
  P.}, \bibinfo{year}{2018}.
\newblock \bibinfo{title}{{Optimal Subsampling for Large Scale Logistic
  Regression}}.
\newblock \bibinfo{journal}{Journal of the American Statistical Association}
  \bibinfo{volume}{113}, \bibinfo{pages}{829--844}.
%Type = Article
\bibitem[{Wang(2013)}]{wang2013examining}
\bibinfo{author}{Wang, I.J.}, \bibinfo{year}{2013}.
\newblock \bibinfo{title}{Examining the full effects of landscape heterogeneity
  on spatial genetic variation: a multiple matrix regression approach for
  quantifying geographic and ecological isolation}.
\newblock \bibinfo{journal}{Evolution} \bibinfo{volume}{67},
  \bibinfo{pages}{3403--3411}.
%Type = Article
\bibitem[{Wang et~al.(2016)Wang, Chen, Gao, Gao and Nie}]{wang2016schatten}
\bibinfo{author}{Wang, Q.}, \bibinfo{author}{Chen, F.}, \bibinfo{author}{Gao,
  Q.}, \bibinfo{author}{Gao, X.}, \bibinfo{author}{Nie, F.},
  \bibinfo{year}{2016}.
\newblock \bibinfo{title}{On the {S}chatten norm for matrix based subspace
  learning and classification}.
\newblock \bibinfo{journal}{Neurocomputing} \bibinfo{volume}{216},
  \bibinfo{pages}{192--199}.
%Type = Article
\bibitem[{Wang et~al.(2017a)Wang, Gao, Gao and Nie}]{wang2017optimal}
\bibinfo{author}{Wang, Q.}, \bibinfo{author}{Gao, Q.}, \bibinfo{author}{Gao,
  X.}, \bibinfo{author}{Nie, F.}, \bibinfo{year}{2017}a.
\newblock \bibinfo{title}{Optimal mean two-dimensional principal component
  analysis with {F}-norm minimization}.
\newblock \bibinfo{journal}{Pattern Recognition} \bibinfo{volume}{68},
  \bibinfo{pages}{286--294}.
%Type = Article
\bibitem[{Wang et~al.(2017b)Wang, Gittens and Mahoney}]{wang2017sketched}
\bibinfo{author}{Wang, S.}, \bibinfo{author}{Gittens, A.},
  \bibinfo{author}{Mahoney, M.W.}, \bibinfo{year}{2017}b.
\newblock \bibinfo{title}{Sketched ridge regression: Optimization perspective,
  statistical perspective, and model averaging}.
\newblock \bibinfo{journal}{Journal of Machine Learning Research}
  \bibinfo{volume}{18}, \bibinfo{pages}{8039--8088}.
%Type = Article
\bibitem[{Wei and De~Pierro(2000)}]{MR1745316}
\bibinfo{author}{Wei, M.}, \bibinfo{author}{De~Pierro, A.R.},
  \bibinfo{year}{2000}.
\newblock \bibinfo{title}{Upper perturbation bounds of weighted projections,
  weighted and constrained least squares problems}.
\newblock \bibinfo{journal}{SIAM Journal on Matrix Analysis and Applications}
  \bibinfo{volume}{21}, \bibinfo{pages}{931--951}.
\newblock \DOIprefix\doi{10.1137/S0895479898336306}.
%Type = Article
\bibitem[{Woodruff et~al.(2014)}]{woodruff2014sketching}
\bibinfo{author}{Woodruff, D.P.}, et~al., \bibinfo{year}{2014}.
\newblock \bibinfo{title}{Sketching as a tool for numerical linear algebra}.
\newblock \bibinfo{journal}{Foundations and Trends{\textregistered} in
  Theoretical Computer Science} \bibinfo{volume}{10}, \bibinfo{pages}{1--157}.
%Type = Article
\bibitem[{Yang et~al.(2016)Yang, Luo, Qian, Tai, Zhang and
  Xu}]{yang2016nuclear}
\bibinfo{author}{Yang, J.}, \bibinfo{author}{Luo, L.}, \bibinfo{author}{Qian,
  J.}, \bibinfo{author}{Tai, Y.}, \bibinfo{author}{Zhang, F.},
  \bibinfo{author}{Xu, Y.}, \bibinfo{year}{2016}.
\newblock \bibinfo{title}{Nuclear norm based matrix regression with
  applications to face recognition with occlusion and illumination changes}.
\newblock \bibinfo{journal}{IEEE Transactions on Pattern Analysis and Machine
  Intelligence} \bibinfo{volume}{39}, \bibinfo{pages}{156--171}.
%Type = Article
\bibitem[{Zhou and Li(2014)}]{zhou2014regularized}
\bibinfo{author}{Zhou, H.}, \bibinfo{author}{Li, L.}, \bibinfo{year}{2014}.
\newblock \bibinfo{title}{Regularized matrix regression}.
\newblock \bibinfo{journal}{Journal of the Royal Statistical Society: Series B
  (Statistical Methodology)} \bibinfo{volume}{76}, \bibinfo{pages}{463--483}.
%Type = Article
\bibitem[{Zhu and Knyazev(2013)}]{ZKny13}
\bibinfo{author}{Zhu, P.}, \bibinfo{author}{Knyazev, A.V.},
  \bibinfo{year}{2013}.
\newblock \bibinfo{title}{Angles between subspaces and their tangents}.
\newblock \bibinfo{journal}{J. Numer. Math.} \bibinfo{volume}{21},
  \bibinfo{pages}{325--340}.

\end{thebibliography}

\end{document}